\documentclass[a4paper,11pt]{amsart}

\usepackage[USenglish]{babel}
\usepackage[utf8]{inputenc}

\usepackage{newtxtext}
\usepackage{amsfonts} 

\usepackage{amsmath} 

\usepackage{hyperref}

\usepackage[alphabetic]{amsrefs} 
\usepackage{amssymb} 
\usepackage{amsthm} 

\usepackage{xcolor} 

\usepackage{fix-cm}

\usepackage{bm}

\usepackage{derivative}

\usepackage{fancyhdr}

\usepackage{array}
\usepackage{braket}
\usepackage{scalefnt}
\usepackage{lipsum}

\usepackage{mathtools}
\usepackage{thmtools}
\usepackage{graphicx}

\usepackage[normalem]{ulem}
\usepackage{tabularx}

\usepackage{paralist}
\usepackage{derivative}
\usepackage{tensor}

\usepackage{tikz}
\usepackage{tikz-cd}
\usepackage{tikz-3dplot}
\usetikzlibrary{cd}
\usetikzlibrary{arrows}
\usetikzlibrary{matrix}
\usetikzlibrary{positioning}

\tikzcdset{arrow style=math font}

\usepackage[erewhon]{newtxmath}
\DeclareMathAlphabet{\mathbb}{U}{msb}{m}{n}

\usepackage{xcolor}
\definecolor{red}{rgb}{1,0,0}
\definecolor{darkred}{RGB}{192,0,0}

\newcommand{\transpose}[1]{\tensor[^{\mathrm{t}}]{#1}{}}

\newcommand{\bfx}{\mathbf{x}}
\newcommand{\bfy}{\mathbf{y}}

\newcommand{\calA}{\mathcal{A}}
\newcommand{\calB}{\mathcal{B}}

\newcommand{\calD}{\mathcal{D}}

\newcommand{\calH}{\mathcal{H}}

\newcommand{\calR}{\mathcal{R}}

\newcommand{\bbC}{\mathbb{C}}

\newcommand{\bbK}{\mathbb{K}}

\newcommand{\bbN}{\mathbb{N}}

\newcommand{\bbP}{\mathbb{P}}

\newcommand{\bbR}{\mathbb{R}}

\newcommand{\bbZ}{\mathbb{Z}}

\newcommand{\mfS}{\mathfrak{S}}

\newcommand{\rme}{\mathrm{e}}

\newcommand{\rmi}{\mathrm{i}}

\renewcommand{\bar}[1]{\overline{#1}}

\newcommand{\Lap}{\mathop{}\!\Delta}

\DeclareMathOperator{\brk}{brk}
\DeclareMathOperator{\Cat}{Cat}

\DeclareMathOperator{\Der}{D}

\DeclareMathOperator{\HF}{HF}

\DeclareMathOperator{\Ker}{Ker}

\DeclareMathOperator{\rk}{rk}

\DeclareMathOperator{\tr}{tr}

\DeclarePairedDelimiter{\pint}{\lfloor}{\rfloor}
\DeclarePairedDelimiter{\pa}{\langle}{\rangle}
\DeclarePairedDelimiter{\pt}{(}{)}

\DeclarePairedDelimiter{\pg}{\{ }{ \} }

\newcommand{\Mat}{\mathrm{Mat}}
\newcommand{\GL}{\mathrm{GL}}
\newcommand{\Oa}{\mathrm{O}}

\newcommand{\SL}{\mathrm{SL}}
\newcommand{\SO}{\mathrm{SO}}

\newcommand{\LT}{\mathrm{LT}}

\newcommand{\mfsl}{\mathfrak{sl}}

\newcommand{\mfso}{\mathfrak{so}}
\usepackage[left=1in, right=1in, top=1in, bottom=1in, includefoot, headheight=13.6pt]{geometry}

\usepackage{enumitem}

\usepackage{adjustbox}

\usepackage{titlesec}
\titleformat{\section}{
 \vspace{2pt}\scshape\fontfamily{ptm}\raggedright\large}{}{0em}{\hspace{-0.4pt}\large \thesection.\hspace{0.5em}}[\color{black}\titlerule \vspace{-1pt}]
 
\titleformat{name=\section,numberless}
{\vspace{2pt}\scshape\fontfamily{ptm}\raggedright\large}{}{0em}
  {\hspace{-0.4pt}\large}[\color{black}\titlerule \vspace{-1pt}]

\allowdisplaybreaks

\numberwithin{equation}{section}
\theoremstyle{definition}
\newtheorem{defn}[equation]{Definition}
\theoremstyle{plain}
\newtheorem{teo}[defn]{Theorem}
\newtheorem{prop}[defn]{Proposition}

\newtheorem{lem}[defn]{Lemma}

\newtheorem*{teor}{Theorem}
\newtheorem{cor}[defn]{Corollary}
\theoremstyle{remark}
\newtheorem{rem}[defn]{Remark}

\allowdisplaybreaks

\usepackage[font=small,labelfont=bf,margin=1.5cm]{caption}

\renewcommand{\sectionmark}[1]{\markboth{\normalfont \scshape\fontfamily{ptm}\selectfont Cosimo Flavi}{\normalfont \scshape\fontfamily{ptm}\selectfont Border rank of powers of ternary quadratic forms}}
\title[\scriptsize Border rank of powers of ternary quadratic forms]{Border rank of powers of ternary quadratic forms}
\author[\scshape Cosimo Flavi]{Cosimo Flavi}
\address{{\normalfont (Cosimo Flavi)},
\normalfont \scshape\fontfamily{ptm}\selectfont
	Università degli Studi di Firenze, 
	Dipartimento di Matematica e Informatica "Ulisse Dini", \normalfont{Viale Giovanni Battista Morgagni 67/a, 50134 Florence, Italy.}}
\email{cosimo.flavi@unifi.it}



\begin{document}
\renewcommand{\abstractname}{\normalfont \scshape\fontfamily{ptm}\selectfont{Abstract}}
\begin{abstract}
We determine the border rank of each power of any quadratic form in three variables. Since the problem for rank $1$ and rank $2$ quadratic forms can be reduced to determining the rank of powers of binary forms, we primarily focus on non-degenerate quadratic forms.  We begin by considering the quadratic form $q_{n}=x_1^{2}+\dots+x_n^{2}$ in an arbitrary number $n$ of variables. We determine the apolar ideal of any power $q_n^s$, proving that it corresponds to the homogeneous ideal generated by the harmonic polynomials of degree $s+1$. Using this result, we select a specific ideal contained in the apolar ideal for each power of a quadratic form in three variables, which, without loss of generality, we assume to be the form $q_3$. After verifying certain properties, we utilize the recent technique of border apolarity to establish that the border rank of any power $q_3^s$ is equal to the rank of its middle catalecticant matrix, namely $(s+1)(s+2)/2$.
\end{abstract}
\maketitle
\thispagestyle{empty}
\section*{Introduction}
\markboth{\normalfont \scshape\fontfamily{ptm}\selectfont Cosimo Flavi}{\normalfont \scshape\fontfamily{ptm}\selectfont Border rank of powers of ternary quadratic forms}
\noindent The problem of determining the minimum integer $r$ such that a homogeneous polynomial of degree $d\in\bbN$ can be expressed as a sum of $r$ different $d$-th powers of linear forms, also known as \textit{Waring rank}, is a classical problem that remains relevant in many applications today. Formally, a homogeneous polynomial can be seen as a symmetric tensor, and the Waring rank corresponds to its symmetric rank. This fact clearly represents a way to approach the classical problem of determining Waring rank of polynomials, referred to as the \textit{Big Waring Problem}. For an overview of tensor rank, border rank, tensor decomposition, and its applications, we refer to \cite{BCC+18} and \cite{Lan12}. Instead, to have a survey on some algorithms or methods that can be used to determine the symmetric rank of a symmetric tensor, one can consult \cite{BGI11} or \cite{LO13}.

In this paper, our focus is on the powers of the quadratic form
\[
q_n=x_1^2+\dots+x_n^2.
\]
Determining the rank of this form is also a classical problem that is related to the Laplace operator. Decompositions of this form as sums of powers of linear polynomials have appeared in various classical papers and books, even dating back to the 19\textsuperscript{th} century (see, for example, \cite{Hou77}{Questions 38-39, p.~129}). Many considerations and examples are presented in \cite{Rez92} by B.~Reznick, who specifically focuses on real decompositions, also known as \textit{representations}. The problem of determining the rank for the case of three variables is still open (with the exception of some specific cases), but in this paper we concentrate on its border rank. To do this, we require an understanding of classical apolarity theory and a relatively recent concept introduced by W.~Buczyńska and J.~Buczyński in \cite{BB21}, called \textit{border apolarity}. Their results provide a method to study the border rank of a tensor in a similar manner to the classic \textit{apolarity lemma} (see,  for example, \cite{IK99}*{Lemma 1.15}).

We begin by recalling some essential elements of apolarity theory in \autoref{section apolarity}, including apolarity action, catalecticant map and apolar ideal. For detailed information, we refer the readers to \cite{IK99}. Then, \autoref{section harmonic polynomials} is focused on the space of harmonic polynomials, which has a special role for the decompositions we are dealing with. In fact, we will prove in \autoref{section apolar ideal}
(refer to \autoref{Teo Apolar ideal}) that the apolar ideal of the form $q_n^s$ corresponds exactly to the ideal generated by harmonic polynomials of degree $s+1$.
Understanding the apolar ideal is essential to establish the minimality of the border rank of the form $q_3^s$. It is a well-known fact that the rank and border rank of a symmetric tensor are always greater or equal to the rank of its most-squared catalecticant matrix. B.~Reznick has already demonstrated in \cite{Rez92} that the rank of each catalecticant matrix of $q_n^s$ is maximum, leading to the following inequalities:
\begin{equation}
\label{relation introduction}
\rk\bigl(q_n^s\bigr)\geq\brk\bigl(q_n^s\bigr)\geq\binom{s+n-1}{n-1}.
\end{equation}
Using border apolarity, we prove that the second inequality in formula \eqref{relation introduction} is an equality for the case of three variables. Specifically, we determine the border rank of every power of any ternary quadratic form, including degenerate forms. In particular, given a degenerate quadratic form $g\in S^2\bbC^3$, if $\rk g=1$, then $\rk \bigl(g^s\bigr)=1$ as well. On the other hand, if $\rk g=2$, we employ binary forms to establish that
\[
\rk\bigl(g^s\bigr)=\brk\bigl(g^s\bigr)=s+1.
\]
We summarize it in the following theorem.
\begin{teor}
Let $g\in S^2\bbC^3$ be an arbitrary ternary quadratic form. For every $s\in\bbN$, if $g$ is degenerate, then 
\[
\brk \bigl(g^s\bigr)=\rk \bigl(g^s\bigr)=\binom{s+\rk g-1}{\rk g-1}.
\]
Otherwise,
\[
\brk\bigl(g^s\bigr)=\binom{s+2}{2}=\frac{(s+1)(s+2)}{2}.
\]
\end{teor}
The strategy we employ to prove this theorem, detailed in \autoref{section border rank}, is based on the aforementioned border apolarity techniques (see \cite{BB21}).
Another concept that we utilize is the \textit{multigraded Hilbert scheme}, introduced by M.~Haiman and B.~Sturmfels in \cite{HS04}. It is  a scheme parametrizing all ideals in a polynomial ring that are homogeneous and have a fixed Hilbert function with respect to a grading by an abelian group. 

Our focus is on the multigraded Hilbert scheme 
$\mathrm{Hilb}^{\scriptscriptstyle{h_{r,n}}}_{{\scriptscriptstyle S[\bbP^n]}}$,
where $S[\bbP^n]$ is the standard homogeneous coordinate ring of $\bbP^n$, and $h_{r,n}$ is defined as
\[
\begin{tikzcd}[row sep=0pt,column sep=1pc]
 h_{r,n}\colon \bbZ\arrow{r} & \bbZ\hphantom{,} \\
  {\hphantom{h_{r,n}\colon{}}} a \arrow[mapsto]{r} & \min\bigl\{\dim S[\bbP^n]_a,r\bigr\},
\end{tikzcd}
\]
representing the Hilbert function of the homogeneous coordinate ring of $r$ generic points in $\bbP^n$. In \cite{BB21}*{Notation 3.14}, W.~Buczyńska and J.~Buczyński define a specific irreducible component of $\mathrm{Hilb}^{\scriptscriptstyle h_{r,n}}_{\scriptscriptstyle S[\bbP^n]}$, denoted by $\mathrm{Slip}_{r,n}$, which corresponds to the closure of the set of the saturated ideals of $r$ distinct points in $X$, denoted by $\mathrm{Sip}_{r,n}$.

According to the statement of the border apolarity theorem (see \cite{BB21}*{Theorem 3.15}, the border rank of a homogeneous polynomial $f$ is at most $r\in\bbN$ if and only if there exists a limit $I\in\mathrm{Slip}_{r,n}$ of ideals of $r$ points such that $I\subseteq f^{\perp}$.

In \autoref{section border rank} we determine such an ideal $I_{s+1}$ for every $s\in\bbN$, satisfying the conditions
\[
I_{s+1}\in\mathrm{Slip}_{r,3},
\] 
where $r=(s+1)(s+2)/2$, and
\[
I_{s+1}\subseteq\bigl(q_3^s\bigr)^{\perp}.
\] 
The structure of this ideal remains similar for every power $s$ and it can be described just introducing the irreducible representations of the Lie algebra $\mfso_3\bbC$ of the special orthogonal group $\SO_3(\bbC)$.

\section{Preliminaries}
\label{section apolarity}
\noindent
Let us consider a finite-dimensional vector space $V$ of dimension $n\in\bbN$ over $\bbC$ and the symmetric algebras
\[
\calR=S\big(V\big),\qquad \calD=S\big(V^*\big).
\]
We can naturally define, for every $d\in\bbN$, a bilinear map
\[
\begin{tikzcd}[row sep=0pt,column sep=1pc]
 \circ\colon S^dV^*\times S^dV\arrow{r} & \bbC\hphantom{,} \\
  {\hphantom{\circ\colon{}}}  (\phi_1\cdots\phi_d, v_1\cdots v_d) \arrow[mapsto]{r} & \displaystyle{\sum_{\sigma\in\mfS_d}} \phi_1(v_{\sigma(1)})\cdots\phi_d(v_{\sigma(d)}),
\end{tikzcd}
\]
also known as the \textit{contraction pairing} or \textit{polar pairing}. This can also be generalized to the $(k,d)$-\textit{partial polarization map}, defined for $k\leq d$ by
\begin{equation}
\label{polarization map}
\begin{tikzcd}[row sep=0pt,column sep=1pc]
 \circ\colon S^kV^*\times S^dV\arrow{r} & S^{d-k}V\hphantom{.} \\
  {\hphantom{\circ\colon{}}}  (\phi_1\cdots\phi_k, v_1\cdots v_d) \arrow[mapsto]{r} & \displaystyle{\sum_{1\leq i_1<\cdots<i_k\leq d}} (\phi_1\cdots\phi_k)\circ \bigl(v_{i_1}\cdots v_{i_k}\bigr)\prod_{j\neq i_1,\dots,i_k}v_j.
\end{tikzcd}
\end{equation}
Once we fix a basis $\{x_1,\dots,x_n\}$ of $V$ and its dual basis $\{y_1,\dots,y_n\}$ of $V^*$, we can identify $S(V)$ and $S(V^*)$ respectively with the polynomial rings
\[
\calR\simeq\bbC[x_1,\dots,x_n],\quad \calD\simeq\bbC[y_1,\dots,y_n].
\]
This allows us to compute the image of each couple of monomials through polarization map:
\[
\bfy^{\alpha}\circ\bfx^{\beta}=
\begin{cases}
\dfrac{\beta!}{(\beta-\alpha)!}\bfx^{\beta-\alpha}\quad &\text{if $\beta-\alpha\geq 0$},\\[2ex]
0\quad &\text{otherwise},
\end{cases}
\]
for every $\alpha,\beta\in\bbN^n$, where we use the notation 
\[
\bfx^{\delta}=x_1^{\delta_1}\dots x_n^{\delta_n},\quad \bfy^{\delta}=y_1^{\delta_1}\dots y_n^{\delta_n},\quad \delta!=\delta_1!\cdots\delta_n!
\]
for every multi-index $\delta=(\delta_1,\dots,\delta_n)\in\bbN^n$ (for further details about apolarity, we refer to \cite{Dol12}*{Chapter 1}). These considerations lead to the identification of the space $\calD$ with the space of polynomial differential operators.
\begin{defn}
Given any homogeneous polynomial $\phi\in \calD_k$, with $k\leq d$, the operator
\[
\begin{tikzcd}[row sep=0pt,column sep=1pc]
 \Der_\phi\colon \calR_d\arrow{r} & \calR_{d-k} \\
  {\hphantom{\Der_\phi\colon{}}} h \arrow[mapsto]{r} & \phi\circ h
\end{tikzcd}
\]
is called the \textit{differential operator} associated to $\phi$.
\end{defn}
The partial polarization map naturally induces an action of the space $\calD$ on the space $\calR$, considering their elements as polynomials.
\begin{defn}
The \textit{apolarity action} of $\calD$ on $\calR$ is defined by extending the polarization maps linearly to each component of $\calD$ and $\calR$, as the map
\[
\begin{tikzcd}[row sep=0pt,column sep=1pc]
 \circ\colon \calD\times\calR\arrow{r} & \calR\hphantom{.} \\
  {\hphantom{\circ\colon{}}}  (\phi,f) \arrow[mapsto]{r} & \Der_\phi(f).
\end{tikzcd}
\]
\end{defn}
\begin{rem}
\label{remark reversing roles}
In the case of contraction pairing, and in general when dealing with dual spaces $\calR$ and $\calD$, we can reverse the roles of variables $x_i$ and $y_i$, establishing an identification between the components $\calR_j$ and $\calD_j$. Thus, we can consider the contraction pairing as a symmetric bilinear form.
\end{rem}

Considering a specific homogeneous polynomial, we can provide the following definition, based on the apolarity action of the space $\calD$ on it.
\begin{defn}
For every homogeneous polynomial $f\in\calR_d$, the \textit{catalecticant map} of $f$ is defined as the linear map
\[
\begin{tikzcd}[row sep=0pt,column sep=1pc]
 \Cat_f\colon \calD\arrow{r} & \calR\hphantom{.} \\
  {\hphantom{\Cat_f\colon{}}} \phi \arrow[mapsto]{r} & \phi\circ f.
\end{tikzcd}
\]
The \textit{apolar ideal} of the polynomial $f$ is defined as the kernel of $\Cat_f$ and corresponds to the set
\[
f^{\perp}=\Set{g\in\calD|g\circ f=0}.
\]
\end{defn}
It is clear that the catalecticant map is graded, meaning that the map
\[
\begin{tikzcd}[row sep=0pt,column sep=1pc]
 \Cat_f^j\colon \calD_j\arrow{r} & \calR_{d-j} \\
  {\hphantom{\Cat_f\colon{}}} g \arrow[mapsto]{r} & g\circ f,
\end{tikzcd}
\]
also referred to as the \textit{$j$-th catalecticant}, is well-defined. The inequality in the next proposition is well known and provides a lower bound for the symmetric rank of a homogeneous polynomial. It is classically attributed to J.~J.~Sylvester (see \cite{Syl51}) and appears in several texts (see e.g.~\cite{Lan12}*{Proposition 3.5.1.1}).
\begin{prop}
\label{prop lower bound}
If $f\in\calR_d$, then
\[
\rk f\geq\brk f\geq\rk\bigl(\Cat_f^j\bigr)
\]
for every $0\leq j\leq d$.
\end{prop}
The theorem of \textit{border apolarity}, described in \cite{BB21} provides an alternative version of the classic apolarity lemma (see \cite{IK99}*{Lemma 1.15}). Their results pertain to all smooth projective toric varieties, but for our purposes, we only require its version for the Veronese variety. Considering the multigraded Hilbert scheme 
$\mathrm{Hilb}^{\scriptscriptstyle{h_{r,n}}}_{{\scriptscriptstyle S[\bbP^n]}}$, we denote by 
\[
\mathrm{Sip}_{r,n}\subseteq\mathrm{Hilb}^{\scriptscriptstyle{h_{r,n}}}_{{\scriptscriptstyle S[\bbP^n]}}
\]
the subset consisting of saturated ideals of $r$ distinct
points in $\bbP^n$ (the abbreviation \textit{sip} stands for \textit{set of ideals of points}), that has a particular topological property in $\mathrm{Hilb}^{\scriptscriptstyle{h_{r,n}}}_{{\scriptscriptstyle S[\bbP^n]}}$.
\begin{prop}[\cite{BB21}*{Proposition 3.13}]
\label{prop unique component}
There is a unique component of the multigraded Hilbert scheme that
contains $\mathrm{Sip}_{r,n}$ as a dense subset.
\end{prop}
The component defined in \autoref{prop unique component} is denoted by $\mathrm{Slip}_{r,n}$ and any ideal belonging to it is a limit of ideals of points (in this case the abbreviation \textit{slip} stands for \textit{scheme of limits of ideals of points}), that is,
\[
\mathrm{Slip}_{r,n}=\overline{\mathrm{Sip}_{r,n}}.
\]
The border apolarity theorem relates the border rank of a homogeneous form to the presence of a limit of ideal of points in its apolar ideal.
\begin{teo}[Border apolarity \cite{BB21}*{Theorem 3.15}]
\label{teo border apolarity}
Let $f\in S^d\bbC^n$ be a homogeneous polynomial of degree $d$. Then $\brk f\geq r$ if and only if there exists a limit $I\in\mathrm{Slip}_{r,n}$ of ideals of $r$-tuples of points such that $I\subseteq f^{\perp}$.
\end{teo}

\section{Harmonic polynomials}
\label{section harmonic polynomials}
\noindent
To continue our analysis of the form $q_n^s$, we need to introduce a particular class of polynomials known as harmonic polynomials. It can be helpful to identify the space $\calD$ with the space of polynomial differential operators, as discussed in \autoref{section apolarity}. 

We observe that the kernel of a homogeneous differential operator is related to the contraction pairing, as it can be seen as a specific orthogonal complement of the space $\phi\calD_{d-k}$.
\begin{prop}
\label{prop ker differential}
Let $\circ\colon\calD_d\times\calR_d\to\bbC$ be the contraction pairing of degree $d\in\bbN$. Then, for every $k\leq d$ and for every polynomial $\phi\in \calD_{k}$,
\[
\Ker(\Der_\phi)=(\phi\calD_{d-k})^\perp\subseteq\calR_d.
\]
\end{prop}
\begin{proof}
For any $g\in \calR_d$, by the argument of \autoref{remark reversing roles}, we have $g\in(\phi\calD_{d-k})^\perp$ if and only if, for every $\psi\in \calD_{d-k}$,
\[
g\circ (\phi\psi)=(\phi\psi)\circ g=0,
\] 
which implies 
\[
\psi\circ(\phi\circ g)=0.
\] 
This means that 
\[
\Der_\phi(g)=\phi\circ g=0,
\]
and hence $g\in\Ker(\Der_\phi)$.
\end{proof}
\autoref{prop ker differential} allows us to introduce the space of harmonic polynomials using the contraction pairing.
\begin{defn}
The space
\[
\calH_n^d=(q_n\calD_{d-2})^\perp=\Ker(\Der_{q_n})\subseteq \calR_d
\]
is called the space of the \textit{$d$-harmonic polynomials} or \textit{$d$-harmonic forms}.
The differential operator $\Der_{q_n}$, also denoted by $\Delta$, is referred to as the \textit{Laplace operator} and it corresponds to the operator $\Delta\colon \calR_d\to \calR_{d-2}$, defined by
\[
\Delta=\sum_{i=1}^n\pdv[2]{}{x_i}.
\]
\end{defn}
\begin{rem}
\label{Rem Euler's formula}
We can make some observations on the behavior of the Laplace operator on polynomials. 
First, for any two forms $g_1,g_2\in \calR$, it can be verified using Leibniz's rule for derivations that
\[
\Delta(g_1g_2)=\Delta(g_1)g_2+g_1\Delta(g_2)+2\sum_{j=1}^n\pdv{g_1}{x_j}\pdv{g_2}{x_j}.
\]
In particular, by recalling \textit{Euler's formula} for a polynomial $f\in S^k\bbC^n$, given by
\[
\sum_{j=1}^n x_j\pdv{f}{x_j}=kf,
\]
we obtain
\begin{align}
\label{rel_Laplace_on_q_n^s}
\Lap\big(q_n^s\big)&=\sum_{j=1}^n\pdv[2]{q_n^s}{x_j}=2s\sum_{j=1}^n\pdv*{\bigl(x_jq_n^{s-1}\bigr)}{x_j}=2s\sum_{j=1}^n\biggl(q_n^{s-1}+x_j\pdv{q_n^{s-1}}{x_j}\biggr)\\
\nonumber&=2snq_n^{s-1}+4s(s-1)\sum_{j=1}^nx_j^2q_n^{s-2}=2s\bigl(n+2(s-1)\bigr)q_n^{s-1}
\end{align}
for every $s\geq 2$.
\end{rem}
The role of harmonic polynomials is crucial in determining a decomposition of the space $\calR_d$. It is well known that the space $\calH_n^d$ is an irriducible $\SO_n(\bbC)$-module (see \cite{GW98}*{Theorem 5.2.4}) for every $d\in\bbN$. The following proposition, whose validity can be verified in \cite{GW98}, provides a decomposition of $\calR_d$ as a direct sum of irreducible representations.
\begin{prop}[\cite{GW98}*{Corollary 5.2.5}]
\label{prop decomposizione armonici}
Let be $d\in\bbN$. Then
\[
\calR_d=q_n\calR_{d-2}\oplus \calH_n^d.
\]
More precisely, we have
\[
\calR_d=\bigoplus_{j=0}^{\left\lfloor\frac{d}{2}\right\rfloor}q_n^j\calH_n^{d-2j}.
\]
\end{prop}
By \autoref{prop decomposizione armonici}, we can easily determine the dimension of each component of the vector space of harmonic polynomials. 
\begin{cor}
\label{cor dimens harm}
For every $d,n\in\bbN$,
\[
\dim\calH_n^d=\dim S^d\bbC^n-\dim S^{d-2}\bbC^n=\binom{d+n-1}{n-1}-\binom{d+n-3}{n-1}.
\]
\end{cor}

\section{The apolar ideal of \texorpdfstring{$q_n^s$}{qns}}
\label{section apolar ideal}
\noindent Now we will prove that the apolar ideal of $q_n^s$ is precisely the ideal generated by harmonic polynomials of degree $s+1$. To do this, we start by observing the behavior of the apolarity action of an arbitrary polynomial on $q_n^s$. By considering monomials of degree $1$ first, we obtain for every $j=1,\dots,n$
\[
y_j\circ q_n^s=\pdv{q_n^s}{x_j}=2sq_n^{s-1}x_j.
\]
Using Leibniz's rule, we can extend this to any polynomial $g\in \calD_k$, with $k\leq s$, and we get
\[
g\circ q_n^s=2^{k}\frac{s!}{(s-k)!}q_n^{s-k}g+q_n^{s-k+1}h,
\]
for a suitable $h\in \calR_{k-2}$ (here, as previously, we use the same notation $g$ for the polynomial obtained replacing $y_j$ by $x_j$ for every $j=1,\dots,n$).  

Now, considering the action of $G=\GL_n(\bbC)$ on both $\calR$ and $\calD$, we see that they naturally have a structure of $G$-modules.
 Next result is related to the catalecticant map and is crucial for our purposes. For any polynomial $h\in\calD$ we denote by $G_h$ the stabilizer of $h$, i.e., the subgroup of $G$ given by
\[
G_h=\set{A\in G|A.h=h}.
\]
\begin{prop}
\label{Prop polar map equivariant}
Let $h\in \calR_d$ and let $G=\GL_n(\bbC)$. Then the catalecticant map of $h$ is $G_h$-equivariant, that is
\[
\Cat_h(A.f)=A.\Cat_h(f)
\]
for every $A\in G_h$.
\end{prop}
\begin{proof}
The apolarity action is $G$-equivariant (see \cite{IK99}*{Proposition A.3}), which means that for any $\phi\in \calD$ and $f\in\calR$, we have
\[
B.(\phi\circ f)=(B.\phi)\circ(B.f)
\]
for every $B\in \GL_n(\bbC)$. Thus, for any $A\in G_h$, we have
\[
\Cat_h(A.\phi)=(A.\phi)\circ h=(A.\phi)\circ (A.h)=A.(\phi\circ h)=A.\Cat_h(\phi)
\]
for every $\phi\in \calD$.
\end{proof}

Now, the next step is to determine the components of lower degree. The fact that the catalecticant matrices of $q_n^s$ are all full-rank has already been shown by B.~Reznick, using \cite{Rez92}*{Theorem 8.15} and referring to \cite{Rez92}*{Theorem 3.7} and \cite{Rez92}*{Theorem 3.16}. Another proof is provided by F.~Gesmundo and J.~M.~Landsberg in \cite{GL19}*{Theorem 2.2}. For the convenience of the reader we give another proof here, as the strategy we use introduces a way to determine the higher components of the apolar ideal of $q_n^s$.
\begin{prop}
\label{Prop Elementi di grado min s}
For every $m\leq s$,
\[
\big(q_n^s\big)^\perp_m=0.
\]
\end{prop}
\begin{proof}
Let $g\in\calD_m$ be a form of degree $m$ such that $g\neq 0$. By the decomposition in \autoref{prop decomposizione armonici}, we can write $g$ uniquely as
\[
g=\sum_{j=0}^{\left\lfloor\frac{m}{2}\right\rfloor}q_n^jh_j,
\]
where $h_j\in\calH_n^{m-2j}$ for every $j=1,\dots,\left\lfloor\frac{m}{2}\right\rfloor$. Considering the contraction action of $g$ on $q_n^s$, by formula \eqref{rel_Laplace_on_q_n^s}, we have:
\begin{equation*}
g\circ q_n^s=\sum_{j=0}^{\left\lfloor\frac{m}{2}\right\rfloor}\big(q_n^jh_j\big)\circ q_n^s=\sum_{j=0}^{\left\lfloor\frac{m}{2}\right\rfloor}h_j\circ \Delta^j\big(q_n^s\big)=h_0\circ q_n^s+2s\bigl(n+2(s-1)\bigr)\sum_{j=1}^{\left\lfloor\frac{m}{2}\right\rfloor}h_j\circ \Delta^{j-1}\big(q_n^{s-1}\big).
\end{equation*}
Thus, by iterating the process and defining the natural number $t$ to be the lowest value in $\pg*{1,\dots,\pint*{\frac{m}{2}}}$ such that $h_{t}\neq 0$, the quantity above can be written as
\[
g\circ q_n^s=\sum_{j=t}^{\pint*{\frac{m}{2}}}C_jh_j\circ q_n^{s-j}=\sum_{j=t}^{\pint*{\frac{m}{2}}}C_jq_n^{s-m+j}h_j+q_n^{s-m+j+1}p_j,
\]
where $p_j\in \calR_{m-2(j+1)}$ and $C_j$ is a constant value equal to
\[
C_j=2^j\frac{s!}{(s-j)!}\biggl(\prod_{l=1}^{j}\bigl(n+2(s-l)\bigr)\biggr)
\]
for every $j=1,\dots,\pint*{\frac{m}{2}}$.
So, we can write 
\begin{align*}
g\circ q_n^s&= q_n^{s-m+t}\Bigg(h_t+q_np_t+\sum_{j=t+1}^{\pint*{\frac{m}{2}}}C_jq_n^{j-t}h_j+q_n^{j-k+1}p_j\Bigg)\\
&=q_n^{s-m+t}\left(h_t+q_n\Bigg(p_t+\sum_{j=t+1}^{\pint*{\frac{m}{2}}}C_jq_n^{j-t-1}h_j+q_n^{j-t}p_j\Bigg)\right).
\end{align*}
Hence, if we had $g\circ q_n^s=0$, then we should have
\[
h_t=-q_n\Bigg(p_t+\sum_{j=t+1}^{\pint*{\frac{m}{2}}}C_jq_n^{j-t-1}h_j+q_n^{j-t}p_j\Bigg),
\]
but this means that $q_n\mid h_k$, that is clearly a contradiction since $h_k$ is harmonic. Therefore, we must have $g\circ q_n^s\neq 0$, and we have proved that in the apolar ideal $\big(q_n^s\big)^\perp$ there are no polynomials of degree lower than $s+1$.
\end{proof}
The stabilizer of the form $q_n^s$ has a quite simple structure and, in particular, it contains the special orthogonal group $\Oa_n(\bbC)$. 
\begin{lem}
\label{Lem stabilizer q^s}
For every $n,s\in\bbN$ the stabilizer of the form $q_n^s$ with respect to the action of $\GL_n(\bbC)$ is the group
\[
G_{q_n^s}=\bbZ_{s}\times \Oa_n(\bbC).
\]
\end{lem}
\begin{proof}
Given any $A\in\GL_n(\bbC)$, we have
\[
A\cdot q_n^s=A\cdot\bigl(x_1^2+\cdots+x_n^2\bigr)^s=\bigl((A\cdot x_1)^2+\dots+(A\cdot x_n)^2\bigr)^s=(A\cdot q_n)^s
\]
and hence
\[
A\cdot q_n^s=(A\cdot q_n)^s=q_n^s,
\]
that is,
\[
\bigl((A\cdot x_1)^2+\dots+(A\cdot x_n)^2\bigr)^s=\bigl(x_1^2+\cdots+x_n^2\bigr)^s.
\]
This means, considering the $(2s)$-th roots of unity, that
\[
A\cdot q_n=\rme^{\frac{2(j-1)\uppi\rmi}{s}}q_n
\]
for some $j\in\bbN$ such that $1\leq j\leq s$ and this implies that
\[
\rme^{-\frac{2(j-1)\uppi\rmi}{s}}A\cdot q_n=q_n.
\]
In particular, since the stabilizer of the form $q_n$ corresponds to the orthogonal group $\Oa_n(\bbC)$, we have
\[
A=\rme^{\frac{2(j-1)\uppi\rmi}{s}}B
\]
for some $B\in\Oa_n(\bbC)$, which proves the statement.
\end{proof}
Before proving that the set of $(s+1)$-harmonic polynomials generates the apolar ideal, we need to prove another lemma. Let us define the linear forms $u,v\in R_1$ as
\[
u=y_1+\rmi y_2,\qquad v=y_1-\rmi y_2.
\]
\begin{lem}
\label{Rem polinomi armoni u v}
For every $d\in\bbN$, the polynomials $u^d,v^d\in\calD_d$ are harmonic. Moreover, if $d\geq s+1$, then $u^d,v^d\in\big(q_n^s\big)^\perp$.
\end{lem} 
\begin{proof}
We prove the statement only for the linear polynomial $u$, since the case of $v$ is analogous. Let us suppose $d\geq 2$. Then we have
\[
\Delta u^d=\pdv[2]{}{y_1}(y_1+\rmi y_2)^d+\pdv[2]{}{y_2}(y_1+\rmi y_2)^d=d(d-1)(y_1+\rmi y_2)^{d-2}\big(1+\rmi^2\big)=0.
\]
We now observe that
\[
(y_1+\rmi y_2)\circ(x_1+\rmi x_2)=1+\rmi^2=0,
\]
and
\[
u\circ q_n^s=(y_1+\rmi y_2)\circ\big(x_1^2+\dots+x_n^2\big)^s=2s(x_1+\rmi x_2)\big(x_1^2+\dots+x_n^2\big)^{s-1}.
\]
Then, if $d\geq s+1$, we obtain by Leibniz's rule
\[
u^d\circ q_n^s=(y_1+\rmi y_2)^{s+1}\circ\big(x_1^2+\dots+x_n^2\big)^s=2^ss!(y_1+\rmi y_2)^{d-s}\circ(x_1+\rmi x_2)^s=0,
\]
and hence $u^d\in\big(q_n^s\big)^\perp$.
\end{proof}
\begin{rem}
\label{rem prod powers u v}
It is clear that the product of two powers of the polynomials $u$ and $v$ cannot be harmonic. Indeed, given any $l,m\in\bbN$ such that $l,m\geq 1$, we have
\[
u^lv^m=\big(y_1^2+y_2^2\big)(y_1+\rmi y_2)^{l-1}(y_1-\rmi y_2)^{m-1},
\]
which means that $\big(y_1^2+y_2^2\big)\mathrel{\bigm|} u^lv^m$.
\end{rem}
We can finally analyze each component of the catalecticant map to determine its kernel and reveal the structure of the apolar ideal of $q_n^s$.
\begin{prop}
\label{Prop ker polar map}
For every $k\in\bbN$ such that $1\leq k\leq s$, let 
\[
\begin{tikzcd}[row sep=0pt,column sep=1pc]
 \Cat_{q_n^s}^{s+k}\colon\calD_{s+k}\arrow{r} & \calR_{s-k} \\
  {\hphantom{\Cat_{q_n^s}^{s+k}\colon{}}} \phi \arrow[mapsto]{r} & \phi\circ q_n^s
\end{tikzcd}
\]
be the component of degree $s+k$ of the catalecticant map of $q_n^s$. Then,
\begin{equation}
\label{relation prop 3.6}
\big(q_n^s\big)_{s+k}^\perp=\Ker\pt*{\Cat_{q_n^s}^{s+k}}=\bigoplus_{j=0}^{k-1}q_n^j\calH_n^{s+k-2j}.
\end{equation}
\end{prop}
\begin{proof}
If $1\leq k\leq s$, then by \autoref{prop decomposizione armonici} we have
\[
\calD_{s+k}=\bigg(\bigoplus_{j=0}^{k-1}q_n^j\calH_n^{s+k-2j}\bigg)\oplus q_n^{k}\calH_n^{s-k}.
\]
To prove the statement it is sufficient to verify that
\[
q_n^{k}\calD_{s-k}\cap\Ker\Bigl(\Cat_{q_n^s}^{s+k}\Bigr)=\{0\}
\]
and
\[
\bigoplus_{j=0}^{k-1}q_n^j\calH_n^{s+k-2j}\subseteq\Ker\pt*{\Cat_{q_n^s}^{s+k}}.
\]
For the first assertion, let us consider a polynomial $h\in \calD_{s-k}$. Then, as we saw in the proof of \autoref{Prop Elementi di grado min s}, we have
\[
\big(q_n^{k}h\big)\circ q_n^s=h\circ\big(\Delta^{k} q_n^s\big)=C_kh\circ q_n^{s-k},
\]
where
\[
C_k=2^k\frac{s!}{(s-k)!}\biggl(\prod_{l=1}^{k}\bigl(n+2(s-l)\bigr)\biggr).
\]
In particular, by \autoref{Prop Elementi di grado min s}, we have $h\circ q_n^{s-k}=0$ if and only if 
\[
h\in\big(q_n^{s-k}\big)_{s-k}^\perp,
\]
which implies $h=0$.

Before proving the second assertion, we recall that
$\calH_n^d$ is an irreducible $\SO_n(\bbC)$-module for every $d\in\bbN$. Moreover, for every $j\in\bbN$ such that $j\leq k-1$, since by \autoref{Lem stabilizer q^s} we have 
\[
\SO_n(\bbC)\subseteq G_{q_n^j},
\]
also the space 
$q_n^j\calH_n^{s+k-2j}$
is an irreducible $\SO_n(\bbC)$-module.
Now, by \autoref{Prop polar map equivariant}, the catalecticant map 
\[
\Cat_{q_n^s}^{s+k}\colon\calD_{s+k}\to\calR_{s-k}
\]
is $\SO_n(\bbC)$-equivariant. Hence, by Schur's Lemma (see e.g.~\cite{FH91}*{Lemma 1.7}), it follows that the restriction 
\[
\Cat^{s+k}_{q_n^s}\big|_{q_n^j\calH_n^{s+k-2j}}\colon q_n^j\calH_n^{s+k-2j}\to\calR_{s-k}
\] 
must be the zero function or an isomorphism on the image. So, if we take the form 
\[
u=y_1+\rmi y_2\in \calD_1,
\] 
then, using \autoref{Rem polinomi armoni u v}, as
\[
q_n^ju^{s+k-2j}\in q_n^j\calH_n^{s+k-2j},
\] 
we obtain the equality
\[
\bigl(q_n^ju^{s+k-2j}\bigr)\circ q_n^s=C_ju^{s+k-2j}\circ q_n^{s-j}=u^{k-j+1}\circ\bigl(u^{s-j+1}\circ q_n^{s-j}\bigr)=0,
\]
which implies that 
\[
\Cat^{s+k}_{q_n^s}\big|_{q_n^j\calH_n^{s+k-2j}}\equiv 0.
\] 
Therefore, the inclusion is proved.
\end{proof}

We finally have all the necessary elements to describe the apolar ideal $\big(q_n^s\big)^\perp$, which is generated by harmonic polynomials of degree $s+1$.
\begin{teo}
\label{Teo Apolar ideal}
The apolar ideal of the form $q_n^s$ is given by
\[
\big(q_n^s\big)^\perp=\bigl(\calH_n^{s+1}\bigr).
\]
\end{teo}
\begin{proof}
By \autoref{Prop Elementi di grado min s}, we have observed that $\big(q_n^s\big)^\perp_k=0$ for every $1\leq k\leq s$. Therefore, to prove the statement, we simply need to show that
\[
\big(q_n^s\big)_d^\perp=\calH_n^{s+1}\calD_{d-s-1}
\]
for every $d\geq s+1$. Now, let $d=s+k$ for some $k\in\bbN$. Then, by \autoref{Prop ker polar map}, if $1\leq k\leq s+1$, this is equivalent to proving that
\begin{equation}
\label{relation HD}
\bigoplus_{j=0}^{k-1}q_n^j\calH_n^{s+k-2j}=\calH_n^{s+1}\calD_{k-1}.
\end{equation}
Gain by \autoref{Prop ker polar map}, we also have 
\[
\big(q_n^s\big)_{s+1}^\perp=\calH_n^{s+1},
\]
and hence we immediately obtain the inclusion
\begin{equation}
\label{relation first inclusion}
\bigoplus_{j=0}^{k-1}q_n^j\calH_n^{s+k-2j}\supseteq\calH_n^{s+1}\calD_{k-1}.
\end{equation}
The reverse inclusion can be obtained by induction on $k$, separately for odd and even values. If $k=1$, the equality is obvious and we have already seen it. If $k=2$, we have to show that
\[
\calH_n^{s+2}\oplus q_n\calH_n^s\subseteq\calH_n^{s+1}\calD_1.
\]
We consider the polynomials 
\[
u=y_1+\rmi y_2,\qquad v=y_1-\rmi y_2,
\]
and, as shown in \autoref{Rem polinomi armoni u v}, we can write the polynomial $u^{s+2}\in\calH_n^{s+2}$ as 
\[
u^{s+1}u\in\calH_n^{s+1}\calD_1,
\]
thus proving, since $\calH_n^{s+2}$ is an irreducible $\SO_n(\bbC)$-module, that
\[
\calH_n^{s+2}\subseteq \calH_n^{s+1}\calD_1.
\]
Now, by \autoref{rem prod powers u v}, we also have that $u^{s+1}v\notin\calH_n^{s+2}$, and in particular, using inclusion \eqref{relation first inclusion}, we have
\[
u^{s+1}v\in\calH_n^{s+1}\calD_1\subseteq\calH_n^{s+2}\oplus q_n\calH_n^s.
\] 
This means that there exist unique forms $h_1\in\calH_n^{s+2}$ and $ h_2\in\calH_n^{s}$, with $h_2\neq 0$ such that
\[
u^{s+1}v=h_1+q_nh_2\in \calH_n^{s+2}\oplus q_n\calH_n^s
\] 
and in particular,
\[
q_nh_2=u^{s+1}v-h_1\in\calH_n^{s+1}\calD_1.
\]
Hence, since $q_n\calH_n^s$ is an irreducible $\SO_n(\bbC)$-module, we have
\[
q_n\calH_n^s\subseteq\calH_n^{s+1}\calD_1,
\]
providing the required inclusion. Now, for $3\leq k\leq s$, assuming the equality is true for $k-2$, we need to show that  
\[
\bigoplus_{j=0}^{k-1}q_n^j\calH_n^{s+k-2j}=\calH_n^{s+k}\oplus q_n\pt*{\bigoplus_{j=0}^{k-2}q_n^{j-1}\calH_n^{s+k-2-2j}}\subseteq\calH_n^{s+1}\calD_{k-1}.
\]
As previously mentioned, we can consider the polynomial 
$u^{s+k}\in\calH_{s+k}$, which can be written as
\[
u^{s+1}u^{k-1}\in\calH_n^{s+1}\calD_{k-1},
\]
and conclude, once again by irreducibility, that
\[
\calH_n^{s+k}\subseteq\calH_n^{s+1}\calD_{k-1}.
\] 
Therefore, since by inductive hypothesis
\[
\bigoplus_{j=0}^{k-2}q_n^{j-1}\calH_n^{s+k-2-2j}\subseteq\calH_n^{s+1}\calD_{k-3},
\]
we have, as $q_n\in \calD_2$,
\[
q_n\bigg(\bigoplus_{j=0}^{k-2}q_n^{j-1}\calH_n^{s+k-2-2j}\bigg)\subseteq\calH_n^{s+1}\calD_{k-1},
\]
which implies the inclusion
\[
\bigoplus_{j=0}^{k-1}q_n^j\calH_n^{s+k-2j}\subseteq\calH_n^{s+1}\calD_{k-1},
\]
thus proving equality \eqref{relation HD}.
It remains to show that for every $d\geq 2(s+1)$,
\begin{equation}
\label{relation casi alti}
\big(q_n^s\big)^\perp_{d}=\calH_n^{s+1}\calD_{d-s-1}
\end{equation}
and, since by definition we have
\[
\big(q_n^s\big)^\perp_{d}=\calD_{d}
\]
for every $d\geq 2(s+1)$, we simply have to prove that
\[
\calD_{2s+m+1}=\calH_n^{s+1}\calD_{s+m}
\]
for every $m\geq 1$. We have just seen, by formulas \eqref{relation prop 3.6} and \eqref{relation HD}, that 
\[
\big(q_n^s\big)^\perp_{2s+1}=\calH_n^{s+1}\calD_s=\bigoplus_{j=0}^{s}q_n^j\calH_n^{2s-2j+1}.
\]
This implies, by decomposition of \autoref{prop decomposizione armonici}, that
\[
\calH_n^{s+1}\calD_s=\calD_{2s+1},
\]
from which we easily get
\[
\calD_{2s+m+1}=\calD_{2s+1}\calD_{m}=\calH_n^{s+1}\calD_{s}\calD_{m}=\calH_n^{s+1}\calD_{s+m}
\]
for every $m\geq 1$, that corresponds to equality \eqref{relation casi alti}.
\end{proof}

By \autoref{Teo Apolar ideal}, it is quite easy to obtain a lower bound for the rank of the form $q_n^s$. Indeed, as a direct consequence of \autoref{prop lower bound}, we have the following corollary.
\begin{cor}
\label{cor cat lower bound}
For every $n$,$s\in\bbN$
\[
\rk\pt*{q_n^s}\geq\brk\pt*{q_n^s}\geq\binom{s+n-1}{n-1}.
\]
\end{cor}

\section{Border rank in three variables}
\label{section border rank}
\noindent The irreducibility of the space of harmonic polynomials as $\SO_n(\bbC)$-modules allows us to determine a canonical basis for the case of three variables in each component. Considering the Lie algebra of the Lie group $\SL_2(\bbC)$, given by the set
\[
\mfsl_2\bbC=\Set{A\in\Mat_2(\bbC)|\tr A=0},
\]
we know that its only irreducible representations correspond to the spaces of even symmetric powers of $\bbC^2$ (see \cite{FH91}*{Section 11.1} for details). That is,
\[
S^{2d}\bbC^2\simeq\bbC[x,y]_{2d}
\]
for each $d\in\bbN$.
In particular, we can consider a natural basis $\calA_2=\{H,E,F\}$ of $\mfsl_2\bbC$, where
\[
H=\begin{pmatrix}
1&0\\
0&-1
\end{pmatrix},\quad 
E=\begin{pmatrix}
0&1\\
0&0
\end{pmatrix},\quad 
F=\begin{pmatrix}
0&0\\
1&0
\end{pmatrix},
\]
satisfying the relations
\begin{equation}
\label{relation EFH}
[H,E]=2E,\quad [H,F]=-2F,\quad [E,F]=H.
\end{equation}
This basis allows the set of monomials
\[
\calB_{d}={\big\{x^{2d-j}y^{j}\big\}}_{j=0,\dots,2d}
\]
to represent a basis of eigenvectors of the action of $H$ on $S^{2d}\bbC^2$, which is therefore diagonalizable.
Moreover, we have that
\[
E\bigl(x^jy^{2d-j}\bigr)\in\bigl\langle x^{j+1}y^{2d-j-1}\bigr\rangle,\qquad F\bigl(x^{2d-j}y^j\bigr)\in\bigl\langle x^{2d-j-1}y^{j+1}\bigr\rangle,
\]
for every $j=0,\dots,2d-1$.

Now, the reason why we have introduced the Lie algebra $\mfsl_2\bbC$ is connected to a well-known consideration regarding representations of $\SO_3(\bbC)$ (see \cite{FH91}*{Section 18.2} for further details). Let us consider the Lie algebra $\mfso_3\bbC$ of the special orthogonal group $\SO_3(\bbC)$, defined as
\[
\mfso_3\bbC=\Set{A\in\Mat_3(\bbC)|A=-\transpose{A}}.
\]
Since $\calH_3^d$ is an irreducible $\SO_3(\bbC)$-module, then it is also an irreducible representation of the Lie algebra $\mfso_3\bbC$ (see e.g.~\cite{FH91}*{Exercise 8.17}).
So, if we consider the basis $\calA_3=\{H,E,F\}$ of $\mfso_3\bbC$, where
\[
H=\begin{pmatrix}
0 & -2\rmi & 0 \\
2\rmi & 0 & 0 \\
0 & 0 & 0  
\end{pmatrix},\quad
E=\begin{pmatrix}
0 & 0 & -1 \\
0 & 0 & -\rmi \\
1 & \rmi & 0  
\end{pmatrix},\quad
F=\begin{pmatrix}
0 & 0 & 1 \\
0 & 0 & -\rmi \\
-1 & \rmi & 0  
\end{pmatrix},
\]
we observe that relations \eqref{relation EFH} are satisfied. Thus, it means not only that
\[
\mfso_3\bbC\simeq\mfsl_2\bbC,
\]
but also, by the uniqueness of irreducible representations of $\mfsl_2\bbC$, that
\[
\calH_3^d\simeq S^{2d}\bbC^2.
\]
Furthermore, by using the linear change of variables in the space $\calD=\bbC[y_1,y_2,y_3]$, defined by
\[
u=\frac{y_1+\rmi y_2}{2},\quad v=\frac{y_1-\rmi y_2}{2},\quad z=y_3,
\]
we can express the elements of $\calA_3$ as 
\[
H=\begin{pmatrix}
2 & 0 & 0 \\
0 & -2 & 0 \\
0 & 0 & 0  
\end{pmatrix},\quad
E=\begin{pmatrix}
0 & 0 & -2 \\
0 & 0 & 0 \\
0 & 1 & 0  
\end{pmatrix},\quad
F=\begin{pmatrix}
0 & 0 & 0 \\
0 & 0 & 2 \\
-1 & 0 & 0  
\end{pmatrix}.
\]
Thus, we can naturally choose a basis that corresponds to the set of monomials $\calB_d$, identifying it
as a set of harmonic polynomials, given by
\[
\calB_{d}=\pg*{p_{d,j}}_{-d\leq j\leq d},
\]
where, for every $k\geq 0$,
\[
p_{d,k}=\sum_{j=0}^{\pint*{\frac{d-k}{2}}}(-1)^ju^{[k+j]}z^{[d-k-2j]}v^{[j]}
\]
and
\[
p_{d,-k}=\sum_{j=0}^{\pint*{\frac{d-k}{2}}}(-1)^ju^{[j]}z^{[d-k-2j]}v^{[k+j]},
\]
using the notation of the \textit{divided powers}, where we define
\[
u^{[m]}=\frac{1}{m!}u^m,\quad v^{[m]}=\frac{1}{m!}v^m,\quad z^{[m]}=\frac{1}{m!}z^m,
\]
for every $m\in\bbN$.
\begin{rem}
Since we have changed set of variables, to verify that all these are harmonic polynomials, we have to use the Laplace operator in the form
\[
\Delta=\pdv[2]{}{z}+\pdv{}{u,v}.
\]
\end{rem}
These polynomials may not appear particularly nice, but they can potentially be described more clearly in the following diagram:
\[
\begin{gathered}
\adjustbox{scale=0.75,center}{
\begin{tikzcd}[scale=0.5em]
\underset{\color{darkred}p_{2,2}}{\pa*{\dfrac{1}{2}u^{2}}} \ar[r,yshift=2.2,"F"]  & \underset{\color{darkred}p_{2,1}}{\pa*{\dfrac{1}{4}uz}} \ar[l,yshift=-2.2,"E"]\ar[r,yshift=2.2,"F"] & \underset{\color{darkred}p_{2,0}}{\pa*{\dfrac{1}{12}\pt*{z^{2}-2uv}}} \ar[l,yshift=-2.2,"E"]\ar[r,yshift=2.2,"F"] & \underset{\color{darkred}p_{2,-1}}{\pa*{\dfrac{1}{4}vz}} \ar[l,yshift=-2.2,"E"]\ar[r,yshift=2.2,"F"] & \underset{\color{darkred}p_{2,-2}}{\pa*{\dfrac{1}{2}v^{2}}}\ar[l,yshift=-2.2,"E"]\hphantom{.}
\end{tikzcd}}\\
\adjustbox{scale=0.75,center}{
\begin{tikzcd}
\underset{\color{darkred}p_{3,3}}{\pa*{\dfrac{1}{6}u^3}} \ar[r,yshift=2.2,"F"]  & \underset{\color{darkred}p_{3,2}}{\pa*{\dfrac{1}{12}u^2z}} \ar[l,yshift=-2.2,"E"]\ar[r,yshift=2.2,"F"] & \underset{\color{darkred}p_{3,1}}{\pa*{\dfrac{1}{30}u\pt*{z^2-uv}}} \ar[l,yshift=-2.2,"E"]\ar[r,yshift=2.2,"F"] & \underset{\color{darkred}p_{3,0}}{\pa*{\dfrac{1}{120}z\pt*{z^2-6uv}}} \ar[l,yshift=-2.2,"E"]\ar[r,yshift=2.2,"F"] & \underset{\color{darkred}p_{3,-1}}{\pa*{\dfrac{1}{30}v\pt*{z^2-uv}}}\ar[l,yshift=-2.2,"E"]\ar[r,yshift=2.2,"F"]& \underset{\color{darkred}p_{3,-2}}{\pa*{\dfrac{1}{12}v^2z}} \ar[l,yshift=-2.2,"E"]\ar[r,yshift=2.2,"F"] & \underset{\color{darkred}p_{3,-3}}{\pa*{\dfrac{1}{6}v^3}}\ar[l,yshift=-2.2,"E"]\hphantom{.}
\end{tikzcd}}\\
\vdots\\
\adjustbox{scale=0.75,center}{
\begin{tikzcd}
\underset{\color{darkred}p_{n,n}}{\pa*{\dfrac{1}{d!}u^d}} \ar[r,yshift=2.2,"F"]  & \underset{\color{darkred}p_{d,d-1}}{\pa*{\dfrac{1}{2d(d-1)!}u^{d-1}z}} \ar[l,yshift=-2.2,"E"]\ar[r,yshift=2.2,"F"] & \cdots \ar[l,yshift=-2.2,"E"]\ar[r,yshift=2.2,"F"] & \underset{\color{darkred}p_{d,-(d-1)}}{\pa*{\dfrac{1}{2d(d-1)!}v^{d-1}z}} \ar[l,yshift=-2.2,"E"]\ar[r,yshift=2.2,"F"] & \underset{\color{darkred}p_{d,-d}}{\pa*{\dfrac{1}{d!}v^d}}\ar[l,yshift=-2.2,"E"].
\end{tikzcd}}
\end{gathered}
\]

Now, before analyzing ternary forms, let us recall some results about binary forms. In particular, considering the variables
\[
u=x_1+\rmi x_2,\qquad v=x_1-\rmi x_2,
\]
we know from the work of B.~Reznick in \cite{Rez92} that the rank of the binary form 
\[
q_2^s=\bigl(x_1^2+x_2^2\bigr)^s=u^sv^s
\]
is minimal and hence equal to the rank of the central catalecticant map.
\begin{teo}[\cite{Rez92}*{Theorem 9.5}]
\label{teo decom Rez}
The minimal real decompositions of $q_2^s$ are given by
\[
q_2^s=\sum_{j=1}^{s+1}(a_jx_1+b_jx_2)^{2s},
\]
where the $2s+2$ points in the set
\[
\bigl\{\pm(a_j,b_j)\bigr\}_{j=1,\dots,s+1}
\]
are the vertices of a regular $(2s+2)$-gon inscribed in a circle of radius $2r(s)$, where
\[
r(s)=(s+1)^{-\frac{1}{2s}}\binom{2s}{s}^{-\frac{1}{2s}}.
\]
\end{teo}
\autoref{teo decom Rez} shows that the minimal real decompositions of $q_2^s$ are essentially unique up to an orthogonal transformation. In particular, we can write 
\[
q_2^s=\sum_{j=1}^{s+1}\big(2r(s)\cos(\tau_{j}) x_1+2r(s)\sin(\tau_{j})x_2\big)^{2s}=\sum_{j=1}^{s+1}\big(r(s)e^{\rmi\tau_{j}} u+r(s)e^{-\rmi\tau_{j}}v\big)^{2s},
\]
where 
\[
\tau_j=\frac{(j-1)\uppi}{s+1}
\]
for every $j=1,\dots,s+1$.
By apolarity, we can easily extend this result and obtain the uniqueness of decompositions even in the complex case. Indeed, by \autoref{Rem polinomi armoni u v} and \autoref{cor dimens harm}, we know that
\[
\bigl(q_2^s\bigr)^{\perp}=\bigl(u^{s+1},v^{s+1}\bigr).
\]
Hence, by the classical apolarity lemma (\cite{IK99}*{Lemma 1.15}), every decomposition of $q_2^s$ corresponds to a polynomial with distinct roots in its apolar ideal. Therefore, each linear combination of $u^{s+1}$ and $v^{s+1}$, up to scalars, can be written as
\[
f_{s+1}=u^{s+1}-e^{i(\theta+\rmi k)}v^{s+1},
\]
for some $\theta\in[0,2\uppi)$ and $k\in\bbR$ (with the only exception of $u^{s+1}$ and $v^{s+1}$). The roots of these polynomials are given by the projective points
\[
[u_j:v_j]=\bigl[e^{\rmi w_{k,\theta,j}}:e^{-\rmi w_{k,\theta,j}}\bigr],\quad j=1,\dots,s+1
\]
where
\[
w_{k,\theta,j}=\frac{2(j-1)\uppi+\theta+\rmi k}{2(s+1)}
\]
for every $j=1,\dots,s+1$. By solving a linear system, we obtain
\[
q_2^s=\sum_{j=1}^{s+1}\big(2r(s)\cos(w_{k,\theta,j}) x_1+2r(s)\sin(w_{k,\theta,j})x_2\big)^{2s}=\sum_{j=1}^{s+1}\big(r(s)e^{\rmi w_{k,\theta,j}} u+r(s)e^{-\rmi w_{k,\theta,j}}v\big)^{2s}.
\]
Thus, we have established the following theorem.
\begin{teo}
\label{minimal decomposition 2 var}
The minimal decompositions of $q_2^s$ are essentially unique, up to an orthogonal complex transformation, and are given by
\[
q_2^s=\sum_{j=1}^{s+1}\big(2r(s)\cos(\tau_{j}) x_1+2r(s)\sin(\tau_{j})x_2\big)^{2s}=\sum_{j=1}^{s+1}\big(r(s)e^{\rmi\tau_{j}} u+r(s)e^{-\rmi\tau_{j}}v\big)^{2s},
\]
where 
\[
r(s)=(s+1)^{-\frac{1}{2s}}\binom{2s}{s}^{-\frac{1}{2s}}.
\]
and
\[
\tau_j=\frac{(j-1)\uppi}{s+1}
\]
for every $j=1,\dots,s+1$.
\end{teo}

\begin{figure}
\center
\begin{tikzpicture}
\begin{scope}[shift={(-3cm,0cm)}]
\draw[thick,->] (-2,0) -- (2,0) node[anchor=north east]{$x_1$};
\draw[thick,->] (0,-2) -- (0,2) node[anchor=north west]{$x_2$};
\foreach \x in {0,45,...,315}{
\pgfpathmoveto{\pgfpointpolar{\x}{1.2cm}}
\pgfpathlineto{\pgfpointpolar{\x+45}{1.2cm}}
\color{blue}
\pgfsetdash{{3pt}{2pt}}{0pt}
\pgfusepath{fill,stroke}}
\foreach \x in {0,45,...,315}
{\pgfpathcircle{\pgfpointpolar{\x}{1.2cm}}{2pt}
\color{blue}
\pgfusepath{fill}};
\end{scope}
\begin{scope}
\draw
node [shift={(-3cm,-3cm)},text width=5.3cm,rounded corners,fill=gray!10,inner sep=1ex,scale=0.7]
{Decomposition of $q_2^3$ (4 points):\\
roots of $f_4=u^4-v^4$.
};
\draw
node [shift={(3cm,-3cm)},text width=5.3cm,rounded corners,fill=gray!10,inner sep=1ex,scale=0.7]
{Decomposition of $q_2^4$ (5 points):\\
roots of $f_5=u^5-v^5$.
};
\end{scope}
\begin{scope}[shift={(3cm,0cm)}]
\draw[thick,->] (-2,0) -- (2,0) node[anchor=north east]{$x_1$};
\draw[thick,->] (0,-2) -- (0,2) node[anchor=north west]{$x_2$};
\foreach \x in {0,36,...,324}{
\pgfpathmoveto{\pgfpointpolar{\x}{1.2cm}}
\pgfpathlineto{\pgfpointpolar{\x+36}{1.2cm}}
\color{darkred}
\pgfsetdash{{3pt}{2pt}}{0pt}
\pgfusepath{fill,stroke}}
\foreach \x in {0,36,...,324}
{\pgfpathcircle{\pgfpointpolar{\x}{1.2cm}}{2pt}
\color{darkred}
\pgfusepath{fill}};
\end{scope} 
\end{tikzpicture}
\caption{Examples of decompositions for the polynomials $q_2^3$ and $q_2^4$. The blue octagon on the left represents the $4$ projective points obtained by the roots of the polynomial $u^4-v^4$, while the red decagon on the right represents the $5$ projective points obtained by the roots of the polynomial $u^5-v^5$ (points opposite to the origin represent the same point in $\bbP^1(\bbC)$).}
\end{figure}
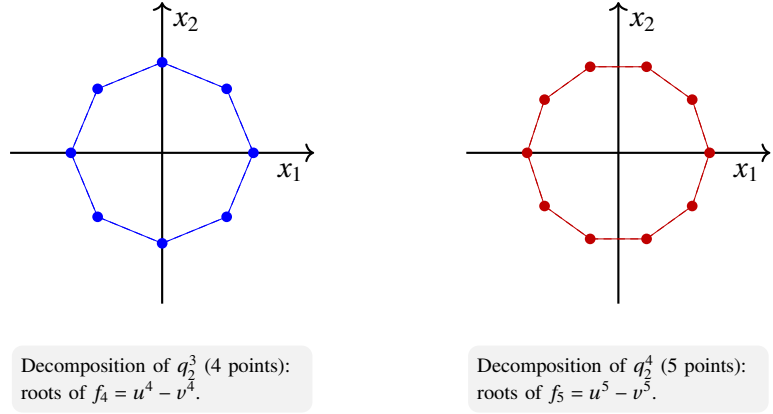

We now have all the tools to determine the border rank of powers of every quadratic form in three variables. We aim to prove the following theorem.
\begin{teo}
\label{teo border rank}
Let $g\in S^2\bbC^3$ be an arbitrary ternary quadratic form. For every $s\in\bbN$, if $g$ is degenerate, then 
\[
\brk \bigl(g^s\bigr)=\rk \bigl(g^s\bigr)=\binom{s+\rk g-1}{\rk g-1}.
\]
Otherwise,
\[
\brk\bigl(g^s\bigr)=\binom{s+2}{2}=\frac{(s+1)(s+2)}{2}.
\]
\end{teo}
In the case where a quadratic form is degenerate, it means that, according to Sylvester's law, it can be expressed, after a linear change of coordinates, as a quadratic form in two variables. Therefore, in this case, being degenerate is equivalent to not being concise.
As mentioned earlier, in order to prove \autoref{teo border rank}, we employ border apolarity. However, to establish this, we first need to determine an ideal that belongs to $\mathrm{Slip_{r,n}}$.

Next, we recall the classical notion of the \textit{Hilbert scheme} parametrizing $r$ points of the projective space $\bbP^n$. This concept was initially developed by A.~Grothendieck (see \cite{Gro61}). Essentially, it is a scheme that parametrizes all the closed subschemes of $\bbP^n$ of length $r$. 

It is important to note the distinction from the multigraded Hilbert scheme $\mathrm{Hilb}^{\scriptscriptstyle{h_{r,n}}}_{{\scriptscriptstyle S[\bbP^n]}}$, defined in \autoref{section apolarity}, which parametrizes all homogeneous ideals of $S[\bbP^n]$, with a fixed Hilbert function $h_{r,n}$. In particular, $\mathrm{Hilb}^{\scriptscriptstyle{h_{r,n}}}_{{\scriptscriptstyle S[\bbP^n]}}$ may contain both saturated and non-saturated ideals (for further details about Hilbert schemes, refer to \cite{EH00}*{Section VI.2.2}).

Nevertheless, we need to consider classical Hilbert schemes due to a specific observation made by T.~Mańdziuk in \cite{Man22}. Let $\bm{\calH ilb}_r\bigl(\bbP^n\bigr)$ denote the Hilbert scheme of the set of $r$-tuples of points in $\bbP^n$. The result of T.~Mańdziuk is applicable to the case of the projective plane $\bbP^2$, whose Hilbert scheme $\bm{\calH ilb}_r\bigl(\bbP^2\bigr)$ is smooth and irreducible (see \cite{Fog68}*{Theorem 2.4}). 
\begin{prop}[\cite{Man22}*{Remark 4.7}]
\label{prop Slip}
If $\bm{\calH ilb}_r\bigl(\bbP^n\bigr)$ is irreducible, $I$ is a saturated ideal and 
\[
I\in\mathrm{Hilb}^{\scriptscriptstyle{h_{r,n}}}_{{\scriptscriptstyle S[\bbP^n]}},
\]
then $I\in\mathrm{Slip}_{r,n}$.
\end{prop}

Therefore, to demonstrate that $I_{d}\in\mathrm{Slip}_{r,n}$, it suffices to prove that $I_d$ is a saturated ideal,where
\[
I_{d}=(p_{d,d},p_{d,d-1},\dots,p_{d,0}).
\]
To establish this, we start by demonstrating that its leading ideal $\LT(I_d)$ is saturated. For this purpose, we require the following lemma, which represents a well-known fact, concerning the ideal of a fat point (refer to \cite{Gim89} for further details on fat points).
\begin{lem}
\label{lem monomial ideal saturated n=3}
For every $s\in\bbN$, the monomial ideal
\[
J_d=\bigl(u^d,u^{d-1}z,\dots,uz^{d-1},z^d\bigr)\subseteq\bbK[u,z,v],
\]
corresponding to the ideal generated by all the monomials of degree $d$ in the variables $u$ and $z$, is saturated.
\end{lem}
Next, we need to prove that the ideal $J_d$ of \autoref{lem monomial ideal saturated n=3} is indeed the leading ideal of $I_d$. To demonstrate this, we simply need to show that the chosen generators form a Gr\"obner basis of $I_d$. The classical method for proving this is to use Buchberger's criterion, which is presented in its standard version in \cite{Eis95}*{Theorem 15.8}. However, we will employ a specific version of this criterion, described accurately by W.~Decker and F.-O.~Schreyer in \cite{DS09}*{Theorem 2.3.9}, which significantly simplifies the verification process.
\begin{teo}[Buchberger's criterion]
Let $f_1,\dots,f_r\in\bbK[x_1,\dots,x_n]$ and let $I=(f_1,\dots,f_r)$. For every $j=2,\dots,r$, let $M_j$ denote the ideal
\[
M_j=\bigl(\LT(f_1),\dots,\LT(f_{j-1})\bigr):\bigl(\LT(f_j)\bigr).
\]
Then $\{f_1,\dots,f_r\}$ is a Gr\"obner basis for $I$ if and only if for every $j=2,\dots,r$ and each minimal generator $x^\alpha$ of $M_j$, the remainder of the division of $x^\alpha f_j$ by $f_1,\dots,f_r$ is equal to zero.
\end{teo}
We can utilize this criterion to establish that the ideal $J_d$ of \autoref{lem monomial ideal saturated n=3} is indeed the leading ideal of $I_d$.
\begin{lem}
For every $s\in\bbN$, the set 
\[
\calB_{d}=\pg*{p_{d,0},\dots,p_{d,d}}
\]
is a Gr\"obner basis of the ideal $I_{d}$ with respect to the lexicographic monomial order defined by the relation 
\[
z>_{\text{lex}}u>_{\text{lex}}v.
\]
\end{lem}
\begin{proof}
By Buchberger's criterion, it suffices to prove that, for every $k=0,\dots,s$, considering the colon ideal
\[
M_k=\bigl(z^d,\dots,z^{d-k+1}u^{k-1}\bigr): \bigl(z^{d-k}u^{k}\bigr)=(z),
\]
the remainder of $zp_{d,k}$ divided by $p_{d,0},\dots,p_{d,d}$ is equal to zero. If $k=s$, then we have
\[
zp_{d,d}=zu^d=up_{d,d-1}
\]
and the statement holds. So, given $k\leq d-1$, since
\[
zp_{d,k}=z\sum_{j=0}^{\pint*{\frac{d-k}{2}}}(-1)^ju^{[k+j]}z^{[d-k-2j]}v^{[j]}=\sum_{j=0}^{\pint*{\frac{d-k}{2}}}(-1)^j(d-k-2j+1)u^{[k+j]}z^{[d-k-2j+1]}v^{[j]},
\]
we can proceed with Buchberger's algorithm by first considering $\LT(h_{d,k-1})=z^{[d-k+1]}u^{[k-1]}$. Indeed we obtain
\[
zp_{d,k}=\frac{\LT(zp_{d,k})}{\LT(p_{d,k-1})}p_{d,k-1}+r_1=\frac{d-k+1}{k}up_{d,k-1}+r_1,
\]
where $r_1$ is the first remainder. In particular, we have
\begin{align*}
r_1&=zp_{d,k}-\frac{d-k+1}{k}up_{d,k-1}\\
&=\sum_{j=0}^{\pint*{\frac{d-k}{2}}}(-1)^j(d-k-2j+1)u^{[k+j]}z^{[d-k-2j+1]}v^{[j]}+\\
&\hphantom{{}={}}{-}\frac{d-k+1}{k}u\sum_{j=0}^{\pint*{\frac{d-k+1}{2}}}(-1)^ju^{[k+j-1]}z^{[d-k-2j+1]}v^{[j]}.
\end{align*}
Now, if $d-k$ is even, then
\[
\pint*{\frac{d-k}{2}}=\pint*{\frac{d-k+1}{2}},
\]
while if it is odd, we observe that the coefficient $d-k-2j-1$ vanishes for 
\[
j=\pint*{\frac{d-k+1}{2}}
\]
and hence we can add a term in the first sum. This allows us to use a formula with only one sum, given by
\begin{align*}
r_1&=\sum_{j=1}^{\pint*{\frac{d-k+1}{2}}}(-1)^j\bigg((d-k-2j+1)-\frac{(d-k+1)(k+j)}{k}\bigg)u^{[k+j]}z^{[d-k-2j+1]}v^{[j]}\\
&=\sum_{j=1}^{\pint*{\frac{d-k+1}{2}}}(-1)^j\frac{(d-k-2j+1)k-(d-k+1)(k+j)}{k}u^{[k+j]}z^{[d-k-2j+1]}v^{[j]}\\
&=\sum_{j=1}^{\pint*{\frac{d-k+1}{2}}}(-1)^{j-1}\frac{(d+k+1)j}{k}u^{[k+j]}z^{[d-k-2j+1]}v^{[j]}\\
&=\frac{d+k+1}{k}v\sum_{j=1}^{\pint*{\frac{d-k+1}{2}}}(-1)^{j-1}u^{[k+j]}z^{[d-k-2j+1]}v^{[j-1]}\\
&=\frac{d+k+1}{k}v\sum_{j=0}^{\pint*{\frac{d-k-1}{2}}}(-1)^{j}u^{[k+j+1]}z^{[d-k-2j-1]}v^{[j]}\\
&=\frac{d+k+1}{k}vp_{d,k+1},
\end{align*}
which guarantees that the last remainder is equal to zero.
\end{proof}
Now, for every ideal $I$, we use the notation $\bar{I}$ to denote the saturation of $I$. To conclude, we need another nice property that T.~Mańdziuk proves in \cite{Man22}, given in the following lemma.
\begin{lem}[\cite{Man22}*{Lemma 4.12}]
\label{lem 4.12Man}
Let $I$ be an ideal in $S[\bbP^n]$ and let $<$ be a monomial order. Then 
\[
\LT_<\big(\bar{I}\big)\subseteq\bar{\LT_<(I)}.
\]
In particular, if $I$ is a homogeneous ideal and $\LT_<(I)$ is a saturated ideal, then $I$ is a saturated ideal.
\end{lem}
We can finally conclude providing the formal proof of \autoref{teo border rank}. 
\begin{proof}[Proof of \autoref{teo border rank}]
If $g$ is a degenerate form such that $\rk g=2$, then the form $g^s$ is equivalent, under the action of $\SL_3(\bbC)$, to the form
\[
q_2^s=\bigl(x_1^2+x_2^2\bigr)^s.
\]
Thus, by \autoref{minimal decomposition 2 var}, we conclude that
\[
\brk \bigl(g^s\bigr)=\rk \bigl(g^s\bigr)=s+1.
\]
If instead $g$ is a non-degenerate form, then $g^s$ is equivalent to the form
\[
q_3^s=\bigl(x_1^2+x_2^2+x_3^2\bigr)^s.
\]
Now, \autoref{lem monomial ideal saturated n=3} and \autoref{lem 4.12Man} together guarantee that the ideal 
\[
I_{s+1}=(p_{s+1,s+1},p_{s+1,s},\dots,p_{s+1,0}),
\] 
is a saturated ideal for every $s\in\bbN$. 
If we now consider the Hilbert function of the ideal $I_{s+1}$ given by
\[
\begin{tikzcd}[row sep=0pt,column sep=1pc]
 \HF_{I_{s+1}}\colon \bbZ\arrow{r} & \bbZ\hphantom{,} \\
  {\hphantom{\HF_{I_{s+1}}\colon{}}} a \arrow[mapsto]{r} & \dim \bigl(\calD_a\big/(I_{s+1})_a\bigr),
\end{tikzcd}
\]
we first observe that for every $k\leq s$, since $(I_{s+1})_k=(0)$, we have
\[
\HF_{I_{s+1}}(k)=\dim\calD_k=\binom{k+2}{2}=\frac{(k+2)(k+1)}{2}.
\]
Moreover, since $I_{s+1}$ is generated by $s+1$ elements of degree $s+1$, we also know that
\[
\HF_{I_{s+1}}(s+1)=\dim \bigl(\calD_{s+1}\big/(I_{s+1})_{s+1}\bigr)=\binom{s+3}{2}-(s+1)=\frac{(s+3)(s+2)-2(s+2)}{2},
\]
which can be simplified to
\[
\HF_{I_{s+1}}(s+1)=\frac{s^2+3s+2}{2}=\binom{s+2}{2}=\HF_{I_{s+1}}(s),
\]
but this means that, since $I_{s+1}$ is a saturated ideal, we have
\[
\HF_{I_{s+1}}(k)=\HF_{I_{s+1}}(s)
\]
for every $k\geq s$. Denoting
\[
r=\binom{s+2}{2},
\]
it follows that
\[
\HF_{I_{s+1}}=h_{r,n}
\]
and hence
\[
I_{s+1}\in\mathrm{Hilb}^{\scriptscriptstyle{h_{r,n}}}_{{\scriptscriptstyle S[\bbP^n]}}.
\]
Therefore, by \autoref{prop Slip}, we have that
\[
I_{s+1}\in\mathrm{Slip}_{r,n},
\]
and finally, using \autoref{teo border apolarity} and \autoref{cor cat lower bound}, we can conclude that
\[
\brk\bigl(q_3^s\bigr)=\binom{s+2}{2}.
\qedhere
\]
\end{proof}

\vspace{20pt}
\subsection*{Acknowledgements.}
This paper has been realized from part of the author's Ph.D.~thesis, undertaken at Alma Mater Studiorum -- Università di Bologna under the patient co-supervision of Alessandro Gimigliano and Giorgio Ottaviani, to whom sincere thanks are due. The author would also reserve special thanks to Fulvio Gesmundo for his valuable suggestions and to Jaros{\l}aw Buczyński and Tomasz Mańdziuk for their precious help in clarifying some issues concerning border apolarity, which have been essential for the realization of this paper.

\end{document}